\documentclass[10pt,bezier]{article}
\usepackage{amsmath,amssymb,amsfonts,euscript,graphicx}

\newtheorem{prethm}{{\bf Theorem}}

\newenvironment{thm}{\begin{prethm}{\hspace{-0.5
               em}{\bf.}}}{\end{prethm}}

\newtheorem{prepro}[prethm]{{\bf Theorem}}

\newtheorem{preprop}[prethm]{{\bf Proposition}}

\newtheorem{precor}[prethm]{{\bf Corollary}}

\newenvironment{cor}{\begin{precor}{\hspace{-0.5
               em}{\bf.}}}{\end{precor}}

\newtheorem{preconj}[prethm]{{\bf Conjecture}}

\newtheorem{preremark}[prethm]{{\bf Remark}}

\newenvironment{remark}{\begin{preremark}\rm{\hspace{-0.5
               em}{\bf.}}}{\end{preremark}}

\newtheorem{preexample}[prethm]{{\bf Example}}

\newenvironment{example}{\begin{preexample}\rm{\hspace{-0.5
               em}{\bf.}}}{\end{preexample}}

\newtheorem{prelem}[prethm]{{\bf Lemma}}

\newenvironment{lem}{\begin{prelem}{\hspace{-0.5
               em}{\bf.}}}{\end{prelem}}

\newtheorem{prelam}{{\bf Lemma}}

\newtheorem{preproof}{{\bf Proof.}}

\newenvironment{proof}[1]{\begin{preproof}{\rm
               #1}\hfill{$\Box$}}{\end{preproof}}

\title{\bf \large The intersection graph of ideals of $\mathbb{Z}_n$ is\\ weakly perfect 
\thanks {2010{ \it Mathematics Subject Classification}: 05C15, 05C69.
 }}
\author{{\normalsize   {\sc R.
Nikandish${}^{\mathsf{}}$}\, and {\sc M.J. Nikmehr${}^{\mathsf{}}$}} \\
 \\
{\footnotesize{${}^{\mathsf{}}$\it Department of Mathematics, K. N. Toosi
University of Technology, Tehran, Iran  }}
{\footnotesize{}}\\
{\footnotesize{}}\\
{\footnotesize{
$\mathsf{r\_nikandish@sina.kntu.ac.ir}$\quad\quad$\mathsf{nikmehr@kntu.ac.ir}$}}}

\date{}
\begin{document}

\maketitle

\begin{abstract}
{\small \noindent A graph is called weakly perfect if its vertex chromatic number equals its
clique number. Let $R$ be a ring and $I(R)^*$ be the set of
all left proper non-trivial ideals of $R$. The intersection graph of ideals
of $R$, denoted by $G(R)$, is a graph with the vertex set $I(R)^*$
and two distinct vertices $I$ and $J$ are adjacent if and only if
$I\cap J\neq 0$. In this paper, it is shown that $G(\mathbb{Z}_n)$, for
every positive integer $n$, is a weakly perfect graph. Also, for some values of $n$, we
give an explicit formula for the vertex chromatic number of $G(\mathbb{Z}_n)$. Furthermore, it is proved
that the edge chromatic number of $G(\mathbb{Z}_n)$ is equal to the maximum degree of $G(\mathbb{Z}_n)$
unless either $G(\mathbb{Z}_n)$ is a null graph with two vertices or a complete graph of odd order.}
\end{abstract}
\vspace{9mm} \noindent{\bf\large 1. Introduction}\\

{\noindent The study of algebraic structures, using the properties of graphs, becomes an exciting
research topic in the last twenty years, leading to many fascinating results and questions.
There are many papers on assigning a graph to a ring, for instance see \cite{ak} and \cite{chakrabarty}. Let $R$ be
a ring with unity. By $I(R)$ and $I(R)^*$, we mean the set of all left ideals of
$R$ and the set of all left proper non-trivial ideals of $R$, respectively.
\\ Let $G$ be a graph with the vertex set $V(G)$. For any $x\in V(G)$, $d(x)$ represents the number of edges incident to $x$, called the
\textit{degree} of the vertex $x$ in $G$. The maximum degree of vertices of $G$ is denoted by $\Delta(G)$.
A graph $G$ is
\textit{connected} if there is a path between every two distinct
vertices. The \textit{complete graph} of
order $n$, denoted by $K_n$, is a graph in which any two distinct
vertices are adjacent.   A \textit{clique} of
$G$ is a maximal complete subgraph of $G$ and the number of
vertices in the largest clique of $G$, denoted by $\omega(G)$, is
called the \textit{clique number} of $G$. For  a graph $G$, let
$\chi(G)$ denote the \textit{vertex chromatic number} of $G$, i.e., the
minimal number of colors which can be assigned to the vertices of
$G$ in such a way that every two adjacent vertices have different
colors. Clearly, for every graph $G$, $\omega(G)\leq \chi(G)$. A graph $G$ is said to be \textit{weakly perfect} if $\omega(G)=\chi(G)$. Recall that
a \textit{$k$-edge coloring} of a graph $G$ is an assignment of $k$ colors $\{1,\ldots,k\}$ to the edges of $G$ such that no two adjacent edges have the same color, and the \textit{edge chromatic number} $\chi'(G)$ of a graph $G$ is the smallest integer $k$ such that $G$ has a $k$-edge coloring. To find some graph coloring methods, we refer the reader to \cite{Kubale}.
\\Let $F=\{S_i| i\in I\}$ be an arbitrary family of sets. The \textit{intersection graph}, $G(F)$, is a graph with $V(G(F))=F$ and two distinct vertices are adjacent if and only if they have non-empty intersection. The following theorem is an interesting fact about intersection graphs due to Marczewski (\cite{Marczewski}).
\begin{thm}\label{Marczewski}
Every simple graph is an intersection graph.
\end{thm}
This result shows that intersection graphs are not weakly perfect in general (for example, the clique number of a cycle with five vertices is 2, but its vertex chromatic number  is 3).
The \textit{intersection graph of ideals} of a ring $R$,
denoted by $G(R)$, is a graph with the vertex set $I(R)^*$ and two
distinct vertices $I$ and $J$ are adjacent if and only if $I\cap
J \neq 0$. This graph was first defined in \cite{chakrabarty}. While the authors were
mainly interested in the study of intersection graph of ideals of $\mathbb{Z}_n$, where $\mathbb{Z}_n$ is the ring of integers modulo $n$. For instance,
they determined the values of $n$ for which the graph of $\mathbb{Z}_n$ is complete, Eulerian or Hamiltonian.
The present article is a natural continuation of the study in this direction. The main aim of
this paper is to show that $G(\mathbb{Z}_n)$ is a weakly perfect graph, for every positive integer $n$.
Moreover, we determine all integers $n$ when $\chi'(G(\mathbb{Z}_n))=\Delta(G(\mathbb{Z}_n))$}.

\vspace{9mm} \noindent{\bf\large 2. Vertex Chromatic Number and Clique Number of $G(\mathbb{Z}_n)$ }\vspace{5mm}

\noindent
In this section, we prove that $\omega(G(\mathbb{Z}_n))=\chi(G(\mathbb{Z}_n))$,
for every positive integer $n$. Also, for some values of $n$, we
give an explicit formula for  $\chi(G(\mathbb{Z}_n))$.

\noindent
Let $n$ be a natural number. Throughout the
paper, without loss of generality, we assume that $n=p_1^{n_1}p_2^{n_2}\ldots p_m^{n_m}$, where $p_i$'s are
 distinct primes and $n_i$'s are natural numbers and $n_1\leq n_2\leq \cdots \leq n_m$. \\
\noindent
We begin with the following remarks.
\begin{remark}\label{1}
Consider the ring $\mathbb{Z}_n$. It follows from Chinese
Remainder Theorem that $\mathbb{Z}_n\cong \mathbb{Z}_{p_1^{n_1}}\times \cdots \times \mathbb{Z}_{p_m^{n_m}}$.
\end{remark}

\begin{remark}\label{2}
$I\in I(\mathbb{Z}_n)$ if and only if $I=I_1\times \cdots \times I_m$, where $I_i \in I(\mathbb{Z}_{p_i^{n_i}})$.
\end{remark}
By Remarks \ref{1} and \ref{2}, one can easily see that $|I(\mathbb{Z}_n)^*|=\prod_{i=1}^m(n_i+1)-2$.

\noindent
Let $I=I_1\times \cdots \times I_m\in I(\mathbb{Z}_n)$. Each $I_i$, $1\leq i \leq m$, is called a component of $I$. Suppose that $k$, $1\leq k \leq m$, is an integer and $F$ is the family of all ideals of $\mathbb{Z}_n$ with exactly $k$ non-zero components in which $I_{i_j}\neq 0$, $1\leq j\leq k$. Obviously,  one can consider $F$ as the set of ideals  with no zero component of the subring $\mathbb{Z}_{p_{i_1}^{n_{i_1}}}\times \cdots \times \mathbb{Z}_{p_{i_k}^{n_{i_k}}}$.\\
\noindent
$\mathbf{Definition.}$
\noindent (i) The number of elements of $F$ is denoted by $\mathcal{W}(F)$.\\
\noindent (ii) Let $G$ be the family of ideals  with no zero component of subring $\mathbb{Z}_{p_{j_1}^{n_{j_1}}}\times \cdots \times \mathbb{Z}_{p_{j_l}^{n_{j_l}}}$. We write $G\subseteq F$, if $\{p_{j_1}^{n_{j_1}},\ldots,p_{j_l}^{n_{j_l}}\}\subseteq \{p_{i_1}^{n_{i_1}},\ldots,p_{i_k}^{n_{i_k}} \}$.\\
\noindent (iii) We write $F\cap G= \varnothing$, if $\{p_{j_1}^{n_{j_1}},\ldots,p_{j_l}^{n_{j_l}}\} \cap\{p_{i_1}^{n_{i_1}},\ldots,p_{i_k}^{n_{i_k}} \}= \varnothing$.\\
\noindent (iv) The family $G$ is said to be complement to $F$, if $F\cap G= \varnothing$ and $\{p_{j_1}^{n_{j_1}},\ldots,p_{j_l}^{n_{j_l}}\} \cup \{p_{i_1}^{n_{i_1}},\ldots,p_{i_k}^{n_{i_k}} \}=\{p_{1}^{n_{1}},\ldots,p_{m}^{n_{m}}\}$ and we write $F^c=G$. Clearly, $(F^c)^c=F$.\\\\
The following example investigate families of ideals of $\mathbb{Z}_n$, when $n=p_1^{n_1}p_2^{n_2}p_3^{n_3}$.
\begin{example}\label{777}
Let $n=p_1^{n_1}p_2^{n_2}p_3^{n_3}$. Then there exists one family of ideals with no zero component. Suppose that $F_0$ is the family of ideals of the form $I_1\times I_2\times I_3$, where each $I_i$ is a non-zero ideal of $\mathbb{Z}_{p_i^{n_i}}$. Let $F_1, F_2,F_3$ be families of ideals of the form
$0\times I_2\times I_3,I_1\times 0\times I_3,I_1\times I_2\times 0$, respectively, where each $I_i$ is a non-zero ideal of $\mathbb{Z}_{p_i^{n_i}}$. Also, assume that $F_4, F_5,F_6$ are families of ideals of the form
$0\times 0\times I_3,0\times I_2\times 0,I_1\times 0\times 0$, respectively, where each $I_i$ is a non-zero ideal of $\mathbb{Z}_{p_i^{n_i}}$. Then
$\mathcal{W}(F_0)=n_1n_2n_3,\mathcal{W}(F_1)=n_2n_3,\mathcal{W}(F_2)=n_1n_3$ and $\mathcal{W}(F_3)=n_2n_3$. Also, $F_i\subset F_0$, for $i=1,\ldots,6$. Moreover, $F_1^c=F_6,F_2^c=F_5$ and $F_4^c=F_3$. If we let $F_7=0\times0\times0$, then $\{F_0,\ldots,F_7\}$ is a partition of ideals of
$\mathbb{Z}_{p_1^{n_1}}\times \mathbb{Z}_{p_2^{n_2}}\times \mathbb{Z}_{p_3^{n_3}}$.
\end{example}
\begin{thm}\label{complete}
Let $n=p_i^{n_i}$. Then $\omega(G(\mathbb{Z}_n))=\chi(G(\mathbb{Z}_n))=n_i-1$.
\end{thm}
\begin{proof}
{It is clear.}
\end{proof}
\begin{thm}\label{last}
The graph $G(\mathbb{Z}_n)$  is weakly perfect, for every $n>0$.
\end{thm}
\begin{proof}
{Let $C=\{G\,|\,\mathcal{W}(G)>\mathcal{W}(G^c)\}\cup A$, where $G\in A$ if and only if $\mathcal{W}(G)=\mathcal{W}(G^c)$ and $|A \cap \{G,G^c\}|=1$. We
show that $C_1=\bigcup_{G\in C}G$ is a clique  of $G(\mathbb{Z}_n)$. Assume to the contrary that there exist ideals $I$ and $J$ in $C_1$ and $I\cap J=0$. Then there exist two families $G_1$ and $G_2$ such that $G_1\cap G_2=\varnothing$. This implies that $G_1\subseteq G_2^c$ and $G_2\subseteq G_1^c$. Consequently,
$\mathcal{W}(G_2)\geq \mathcal{W}(G_2^c)\geq \mathcal{W}(G_1)\geq \mathcal{W}(G_1^c)\geq \mathcal{W}(G_2)$ and so $\mathcal{W}(G_2)= \mathcal{W}(G_2^c)= \mathcal{W}(G_1)= \mathcal{W}(G_1^c)$. Hence, $G_1=G_2^c$ and $G_1, G_2 \in C$, a contradiction. Therefore, $|C_1|\leq \omega(G(\mathbb{Z}_n))$. To complete the proof, we show that $\chi(G(\mathbb{Z}_n))\leq |C_1|$. We color all ideals in $C_1$ with different colors and color each family $H$ of ideals out of $C$ with
colors of ideals of $H^c$. Now, we show that this is a proper vertex coloring of $G(\mathbb{Z}_n)$. Suppose that $I,J$ are adjacent vertices in $G(\mathbb{Z}_n)$. Without loss of generality, one can assume that
there are different families $F_i$ and $F_j$ such that $I\in F_i$ and $J\in F_j$ and at least one of $I,J$
is not contained in $C_1$. Since $F_i\neq F_j$, we deduce that $F_i\neq F_j^c, F_i^c\neq F_j, F_i^c\neq F_j^c$. Therefore, $I$ and $J$
have different colors. Thus we obtain a proper vertex coloring for $G(\mathbb{Z}_n)$, as desired.~
}
\end{proof}

\begin{thm}\label{aval}
Let $n_m\geq \prod_{i=1}^{n_{m-1}}n_i$. Then $$\omega(G(\mathbb{Z}_n))=\chi(G(\mathbb{Z}_n))=n_m\prod_{i=1}^{n_{m-1}}(n_i+1)-1.$$
\end{thm}
\begin{proof}
{It is enough to see that $F\in C$ if for every $J\in F$ the ideal $J$ contains $I=0\times  \cdots \times 0\times (p_m^{n_m-1})$.}
\end{proof}
From the previous theorem we have the following immediate corollaries.
\begin{cor}\label{m2}
Let $n=p_1^{n_1}p_2^{n_2}$. Then $\omega(G(\mathbb{Z}_n))=\chi(G(\mathbb{Z}_n))=n_2(n_1+1)-1$.
\end{cor}
\begin{proof}
{By the assumption, $n_2\geq n_1$. So the result follows from Theorem \ref{aval}.}
\end{proof}
\begin{cor}\label{field}
Let $n_1=n_2=\cdots =n_m=1$. Then $\omega(G(\mathbb{Z}_n))=\chi(G(\mathbb{Z}_n))=2^{m-1}-1$.
\end{cor}
Let $R$ be a ring and $R\cong F_1\times \cdots \times F_m$, where each $F_i$ is a field. By a similar argument in the proof of  Theorem \ref{aval} and Corollary \ref{field}, one can show that $\omega(G(R))=\chi(G(R))=2^{m-1}-1$.
\begin{thm}\label{odd}
Let $m> 1$ be an odd number such that
$\prod_{i=0}^{\lfloor \frac{m}{2}\rfloor-1}n_{m-i}\leq \prod_{i=1}^{m-\lfloor \frac{m}{2}\rfloor}n_i$. Then $\omega(G(\mathbb{Z}_n))=\chi(G(\mathbb{Z}_n))=$ $$|\{I\in I(\mathbb{Z}_{n})^*\,|\,I \rm{\,has \,at \,most
\lfloor \frac{m}{2}\rfloor \,zero \,components}\}|.$$
\end{thm}
\begin{proof}
{It is enough to see that $F\in C$ if for every $J\in F$ the ideal $J$ has at most
$\lfloor \frac{m}{2}\rfloor$ zero components.
}
\end{proof}
\begin{cor}
Let $m$ be an odd number and $n_1=n_2=\cdots =n_m=\alpha$. Then
$$\omega(G(\mathbb{Z}_n))=\chi(G(\mathbb{Z}_n))= \sum_{i=0}^{\lfloor \frac{m}{2}\rfloor} \big(\, ^m_i \,\big)\alpha^{m-i}-1.$$
\end{cor}
\begin{proof}
{By Theorem \ref{odd}, $\omega(G(\mathbb{Z}_n))=\chi(G(\mathbb{Z}_n))=$$$|\{I\in I(\mathbb{Z}_{n})^*|I \rm{\,has \,at \,most \lfloor \frac{m}{2}\rfloor \,zero \,components}\}|.$$ It is clear that there exist
$\big(\, ^m_i \,\big)\alpha^{m-i}$ ideals with $i$ non-zero components. Note that the ideal  $\mathbb{Z}_{p_1^{\alpha}}\times  \cdots \times \mathbb{Z}_{p_m^{\alpha}}$ is not a vertex of $G(\mathbb{Z}_n)$.}
\end{proof}
\noindent In the sequel, similar results in case $m$ is an even number are given.\\
\begin{thm}\label{even}
Let $m>2$ be an even number such that $\prod_{i=0}^{\frac{m}{2}-2}{n_{m-i}}\leq \prod_{i=1}^{\frac{m}{2}+1}{n_i}$. Then $\omega(G(\mathbb{Z}_n))=\chi(G(\mathbb{Z}_n))=$$$|\{I\in I(\mathbb{Z}_{n})^*\,|\,I \rm{\,has \,at \,most  \frac{m}{2}-1 \,zero \,components}\}|+\sum_{F \in A}\mathcal{W}(F).$$
\end{thm}
\begin{proof}
{It is enough to see that $F\in C$ if for every $J\in F$ either $J$ has at most  $\frac{m}{2}-1$ zero components or $J$ has exactly $\frac{m}{2}$ zero components and $\mathcal{W}(F)=\mathcal{W}(F^c)$.
}
\end{proof}

\noindent  We close this section with the following corollary.
\begin{cor}
Let $m$ be an even number and $n_1=n_2=\cdots =n_m=\alpha$. Then
$$\omega(G(\mathbb{Z}_n))=\chi(G(\mathbb{Z}_n))= \sum_{i=0}^{ \frac{m}{2}-1} \big(\, ^m_i \,\big)\alpha^{m-i}+\frac{\big(\, ^m_\frac{m}{2} \,\big)\alpha^{\frac{m}{2}}}{2}-1.$$
\end{cor}
\begin{proof}
{Since $n_1=n_2=\cdots =n_m=\alpha$, it is easily seen that
$\prod_{i=0}^{\frac{m}{2}-2}{n_{m-i}}\leq \prod_{i=1}^{\frac{m}{2}+1}{n_i}$. Moreover, there are
$\big(\, ^m_i \,\big)\alpha^{m-i}$ ideals with $i$ non-zero components. Note that, if $i=\frac{m}{2}$, exactly $\frac{\big(\, ^m_\frac{m}{2} \,\big)\alpha^{\frac{m}{2}}}{2}$ of proper ideals are adjacent together and $\mathbb{Z}_{p_1^{\alpha}}\times \cdots \times \mathbb{Z}_{p_m^{\alpha}}$ is not a vertex of $G(\mathbb{Z}_n)$. The result, now, follows from Theorem \ref{even}.   }
\end{proof}

\noindent It follows from \cite{akbari}, for every ring $R$, if $\omega(G(R))<\infty$, then $\chi(G(R))<\infty$. Also, we have not found any example of a ring $R$ such that $G(R)$ is not weakly perfect. These positive results motivate the following conjecture.\\\\
$\mathbf{Conjecture.}$
\noindent For every ring $R$, $G(R)$ is a weakly perfect graph.

\vspace{9mm} \noindent{\bf\large 3.
Edge Chromatic Number of $G(\mathbb{Z}_n)$  }\vspace{5mm}

\noindent
In this section we study the edge chromatic number of  $G(\mathbb{Z}_n)$
and prove that, for every positive integer $n$, $\chi'(G(\mathbb{Z}_n))=\Delta(G(\mathbb{Z}_n))$, unless $n=pq$, where $p,q$ are distinct primes, or $n=p^{m}$, where $m$ is an even number. If $n=pq$, then $G(\mathbb{Z}_n)$ is a null graph with two vertices and if $n=p^{m}$, where $m$ is an even number, then $G(\mathbb{Z}_n)$ is a complete
graph of order $m-1$. First, we need the following lemmas:

\begin{lem} \rm\cite[p. 16]{Yap}
If $G$ is a simple graph, then either $\chi'(G)=\Delta(G)$ or $\chi'(G)=\Delta(G)+1$.
\end{lem}

\begin{lem} \label{Beine} \rm\cite[Corollary 5.4]{Beineke}
Let $G$ be a simple graph. Suppose that for every vertex $u$ of maximum degree, there exists an edge
$u-\hspace{-.2cm}-v$ such that $\Delta(G)-d(v)+2$ is more than the number of vertices with maximum degree in $G$. Then
$\chi'(G)=\Delta(G)$.
\end{lem}

\begin{lem} \label{Plantholt}\rm\cite[Theorem D]{Plantholt}
If $G$ has order $2s$ and maximum degree $2s-1$, then $\chi'(G)=\Delta(G)$. If $G$ has order $2s+1$
and maximum degree $2s$, then $\chi'(G)=\Delta(G)+1$ if and only if the size of $G$ is at least $2s^2+1$.
\end{lem}

\noindent Now, we are ready to state our main result in this section.

\begin{thm}
Let $n$ be a positive integer. Then $\chi'(G(\mathbb{Z}_n))=\Delta(G(\mathbb{Z}_n))$, unless $n=p_1p_2$, where $p_1,p_2$ are distinct primes, or $n=p^{m}$, where $m$ is an even number.
\end{thm}

\begin{proof}
{Suppose that $n=p_1^{n_1}p_2^{n_2}\ldots p_m^{n_m}$, where $p_i$'s are
all distinct primes and $n_i$'s are all natural numbers. If $m=1$, then  $G(\mathbb{Z}_n)$ is a complete graph
of order $n_1-1$. If $n_1$ is an even number, then  $\chi'(G(\mathbb{Z}_n))=\Delta(G(\mathbb{Z}_n))+1$, otherwise $\chi'(G(\mathbb{Z}_n))=\Delta(G(\mathbb{Z}_n))$. Thus assume that $m\geq 2$. We continue the proof in the following cases:

\noindent Case 1. $n_1=n_2=\cdots=n_m=1$. In this case, $\mathbb{Z}_n\cong F_1\times \cdots \times F_m$, where each $F_i$ is a field. If $m=2$, then $\mathbb{Z}_n\cong F_1\times F_2$ and so $G(\mathbb{Z}_n)$ is a null graph with two vertices. Therefore, $\chi'(G(\mathbb{Z}_n))=\Delta(G(\mathbb{Z}_n))+1$. Then suppose that $m\geq 3$. For each $j=1,\ldots,m$, define the ideal $I_j=F_1\times \cdots \times F_{j-1}\times 0\times F_{j+1} \times \cdots \times F_m$. Then $d(I_j)=\Delta(G(\mathbb{Z}_n))=2^{m}-4$,  for each $1\leq j \leq m$. Also,
$I_1,\ldots,I_m$ are all vertices with maximum degree in $G(\mathbb{Z}_n)$. Let $u=I_i$ ($1\leq i \leq m$) be a vertex of  maximum degree, say $I_1$. Then $v=0\times F_2\times 0 \times\cdots \times 0$ is an adjacent vertex to $u$. Thus
$\Delta(G(\mathbb{Z}_n))-d(v)+2=2^{m}-4-(2^{m-1}-2)+2=2^{m}-2^{m-1}$. Since $m\geq 3$, we conclude that $2^{m}-2^{m-1}> m$, for each $m$. By Lemma \ref{Beine}, $\chi'(G(\mathbb{Z}_n))=\Delta(G(\mathbb{Z}_n))$.

\noindent Case 2. Every $n_i$ is an even number. Then $|V(G(\mathbb{Z}_n))|$ is an odd number of the form
$2s+1$. Note that every vertex of the form $I_1\times \cdots \times I_m$, where $I_i$ is a non-zero ideal of $\mathbb{Z}_{p_i^{n_i}}$ is adjacent to all vertices of $G(\mathbb{Z}_n)$. Since the size of a complete graph of order $2s+1$ is $2s^2+s$, if we prove that the intersection graph $G(\mathbb{Z}_n)$, in this case, losses at least $s$ edges, then by Lemma \ref{Plantholt}, $\chi'(G(\mathbb{Z}_n))=\Delta(G(\mathbb{Z}_n))$. One can consider the ring $\mathbb{Z}_n$ of the form
$R_1\times R_2$, where $R_1,R_2$ are two rings with $|I(R_1)|=\alpha$ and $|I(R_2)|=\beta$. Therefore, $\alpha\beta=2s+3$. Clearly, every vertex of the form $(I, 0)$ is not adjacent to every vertex of the form $(0, J)$. Then $G(R_1 \times R_2)$ losses at least $\alpha\beta-(\alpha+\beta)+1$ edges. It is easily checked that $\alpha\beta-(\alpha+\beta)+1>s=\frac{\alpha\beta-3}{2}$, as desired.

\noindent Case 3. There exists at least one $i$  such that $n_i$ is an odd number and $n_j > 1$ for some $j$.
Therefore $\mathbb{Z}_{p_j^{n_j}}$ has at least one non-trivial proper ideal and then $\mathbb{Z}_n$ has an ideal with no zero complement which is adjacent to every other vertex. Since $|V(G(\mathbb{Z}_n))|$ is an even number, by Lemma \ref{Plantholt}, $\chi'(G(\mathbb{Z}_n))=\Delta(G(\mathbb{Z}_n))$.
 }
\end{proof}

\noindent{\bf Acknowledgements.} The authors would like to express their deep gratitude to the professor Saieed Akbari for
his fruitful comments. Also, we thank to the referee for his/her careful reading and his/her valuable suggestions.

{}


\begin{thebibliography}{}{\small

\bibitem{ak} S. Akbari, D. Kiani, F. Mohammadi, S. Moradi, The total graph and regular graph of a commutative ring, J. Pure Appl. Algebra 213 (12) (2009), 2224--2228.


\bibitem{akbari} S. Akbari, R. Nikandish, M.J. Nikmehr, Some results on the intersection graphs of ideals of  rings, J. Algebra Appl., to appear.

\bibitem{Beineke}L.W. Beineke, B.J. Wilson, Selected Topics in Graph Theory, Academic Press Inc., London, 1978.
\bibitem{chakrabarty}I. Chakrabarty, S. Ghosh, T.K. Mukherjee, M.K.
Sen, Intersection graphs of ideals of rings, Discrete
Math. 309 (17) (2009), 5381--5392.


\bibitem{Kubale}  M. Kubale, Graph Colorings, American Mathematical Society, 2004.

\bibitem{Marczewski} E. Marczewski, Sur deux propri\'{e}t\'{e}s des class d'ensembles,
Fund. Math. 33 (1945), 303--307.
\bibitem{Plantholt} M.J. Plantholt, The chromatic index of graphs with large maximum degree,
Discrete Math. 47 (1981), 91--96.


\bibitem{Yap}H.P. Yap, Some Topics in Graph Theory, in: London Math. Soc. Lecture Note Ser., vol. 108, 1986.

}\end{thebibliography}
\end{document}